\documentclass[12pt]{article}
\usepackage{amsmath, amsthm, amssymb,latexsym,enumerate}
\usepackage{graphicx}
\usepackage{hyperref}
\usepackage{xcolor}
\usepackage{todonotes}
\usepackage{cite}
\usepackage{dsfont}

\newtheorem{theorem}{Theorem}[section]
\newtheorem{lemma}[theorem]{Lemma}

\newtheorem{claim}[theorem]{Claim}

\newtheorem{conjecture}[theorem]{Conjecture}
\newtheorem{question}[theorem]{Question}

\newtheorem*{keylemma}{Key Lemma}
\numberwithin{equation}{section}

\newcommand{\kl}{Key~Lemma}

\DeclareMathOperator{\epn}{epn}

\title{Tight bound for independent domination of cubic graphs without $4$-cycles}

\author{Eun-Kyung Cho\thanks{
Department of Mathematics, Hankuk University of Foreign Studies, Yongin-si, Gyeonggi-do, Republic of Korea.
 \texttt{ekcho2020@gmail.com}
}
\and Ilkyoo Choi\thanks{
Department of Mathematics, Hankuk University of Foreign Studies, Yongin-si, Gyeonggi-do, Republic of Korea.
\texttt{ilkyoo@hufs.ac.kr}
}
\and Hyemin Kwon\thanks{
Department of Mathematics, Ajou University, Suwon-si, Gyeonggi-do, Republic of Korea.
\texttt{khmin1121@ajou.ac.kr}
}
\and Boram Park\thanks{
Department of Mathematics, Ajou University, Suwon-si, Gyeonggi-do, Republic of Korea.
\texttt{borampark@ajou.ac.kr}
}}
\date\today

\begin{document}
 
\maketitle

\begin{abstract} 
Given a graph $G$, a {\it dominating set} of $G$ is a set $S$ of vertices such that each vertex not in $S$ has a neighbor in $S$.  
 The {\it domination number} of $G$, denoted  $\gamma(G)$, is the minimum size of a dominating set of $G$. 
The {\it independent domination number} of $G$, denoted $i(G)$, is the minimum size of a dominating set of $G$ that is also independent. 
 
Recently, Abrishami and Henning proved that if $G$ is a cubic graph with girth at least $6$, then $i(G) \le \frac{4}{11}|V(G)|$.
We show a result that not only improves upon the upper bound of the aforementioned result, but also applies to a larger class of graphs, and is also tight. 
Namely, we prove that if $G$ is a cubic graph without $4$-cycles, then $i(G) \le \frac{5}{14}|V(G)|$, which is tight.
Our result also implies that every cubic graph $G$ without $4$-cycles satisfies $\frac{i(G)}{\gamma(G)} \le \frac{5}{4}$, which  partially answers a question by O and West in the affirmative. 
\end{abstract}

\section{Introduction}

All graphs in this paper are finite and simple.
Given a graph $G$, let $V(G)$ and $E(G)$
denote the vertex set and the edge set, 
respectively, of $G$.
A {\it dominating set} of $G$ is a subset $S$ of $V(G)$ such that each vertex not in $S$ has a neighbor in $S$. 
The {\it domination number} of $G$, denoted $\gamma(G)$, is the minimum size of a dominating set of $G$.
Domination is an extensively studied classic topic in graph theory, to the point that 
there are several books focused solely on domination, see~\cite{book1,book2,book3,book4}.

A dominating set of a graph $G$ that is also an independent set in $G$ is an {\it independent dominating set} of $G$.
The {\it independent domination number} of $G$, denoted $i(G)$, is the minimum size of an independent dominating set of $G$. 
Note that every graph has an independent dominating set, as a maximal independent set is equivalent to an independent dominating set. 
This concept appeared in the literature as early as 1962 by Berge~\cite{berge1962theory} and Ore~\cite{ore1962theory}.
For a survey regarding independent domination, see~\cite{goddard2013independent}.

For a connected $k$-regular graph $G$ where $k\geq 1$, Rosenfeld~\cite{rosenfeld1964independent} showed that $i(G)\leq \frac{|V(G)|}{2}$, which is tight only for the balanced complete bipartite graph $K_{k, k}$. 
The next natural step is to try to improve the upper bound on the independent domination number when the balanced complete bipartite graph is excluded. 
Extending a result by Lam, Shiu, and Sun~\cite{lam1999independent}, Cho, Choi, and Park~\cite{cho2021independent} recently proved the following result: 

\begin{theorem}[\cite{cho2021independent}]\label{thm:kreg}
For $k \ge 3$, if $G$ is a connected $k$-regular graph that is not $K_{k, k}$, then $i(G)\le \frac{k-1}{2k-1}|V(G)|$.
\end{theorem}

Note that the bound in Theorem~\ref{thm:kreg} is tight for $k \in \{3,4\}$ by $C_5\square K_2$ and the $4$-regular expansion of a $7$-cycle, but the tightness of the bound is unknown for $k \ge 5$.
See the two left figures of Figure~\ref{fig:tight} for an illustration.
To our knowledge, this is the best upper bound on the independent domination number not only for cubic graphs, but also for regular graphs in general.

Another popular way to exclude the complete bipartite graph is to impose restrictions on cycle lengths. 
Recall that the {\it girth} of a graph is the length of a smallest cycle in the graph.
According to~\cite{goddard2013independent}, the following conjecture was made by  Verstra\"ete:

\begin{conjecture}[see~\cite{goddard2013independent}]
If $G$ is a cubic graph with girth at least $6$, then $i(G) \le \frac{|V(G)|}{3}$.
\end{conjecture}

The girth condition in the conjecture cannot be reduced by the right graph of Figure~\ref{fig:tight}, as it has 14 vertices, girth 5, and independent domination number $5$. 
Note that Duckworth and Wormald~\cite{duckworth2010} constructed an infinite family of connected cubic graphs with girth $5$ whose independent domination number is exactly $\frac{|V(G)|}{3}$. 
For other results regarding the independent domination number of cubic graphs with girth 5, see~\cite{harutyunyan2021}.

Graphs without odd cycles, namely bipartite graphs, is also an intensively researched class of graphs. 
Goddard et al.~\cite{goddard2012independent} made the following conjecture regarding the independent domination number of bipartite cubic graphs:

\begin{conjecture}[~\cite{goddard2012independent}]\label{conj-bi}
If $G$ is a connected bipartite cubic graph that is not $K_{3,3}$, then $i(G)\leq \frac{4}{11}|V(G)|$.
\end{conjecture}

If the above conjecture is true, then it is tight by a graph displayed in~\cite{goddard2012independent}.
Henning, L\"owenstein, and Rautenbach \cite{henning2014independent} showed that Conjecture~\ref{conj-bi} is true if a $4$-cycle is also forbidden, so the conjecture is true for bipartite cubic graphs with girth at least $6$. 
%
%
Recently, in 2018, Abrishami and Henning~\cite{abrishami2018independent} improved the aforementioned result by showing that the same conclusion holds when the bipartite condition (but keeping the restriction on the girth) is removed. 
The following is the result by Abrishami and Henning:

\begin{theorem}[\cite{abrishami2018independent}]\label{thm:3reggirth6}
If $G$ is a cubic graph with girth at least $6$, then $i(G) \le \frac{4}{11}|V(G)|$.
\end{theorem}

Our main result improves Theorem~\ref{thm:3reggirth6} in a strong sense. 
We not only improve the bound from $\frac{4}{11}$ to $\frac{5}{14}$, but also enlarge the class of graphs the theorem applies to. 
Moreover, the new bound is also tight, by the right graph in Figure~\ref{fig:tight}, which has $14$ vertices and independent domination number $5$. 
We now explicitly state our  main result: 

\begin{theorem}\label{mainthm:cubic}
If $G$ is a cubic graph without $4$-cycles,
then $i(G) \le \frac{5}{14}|V(G)|$.
\end{theorem}

\begin{figure}
    \centering 
    \includegraphics[scale=0.8, page=1]{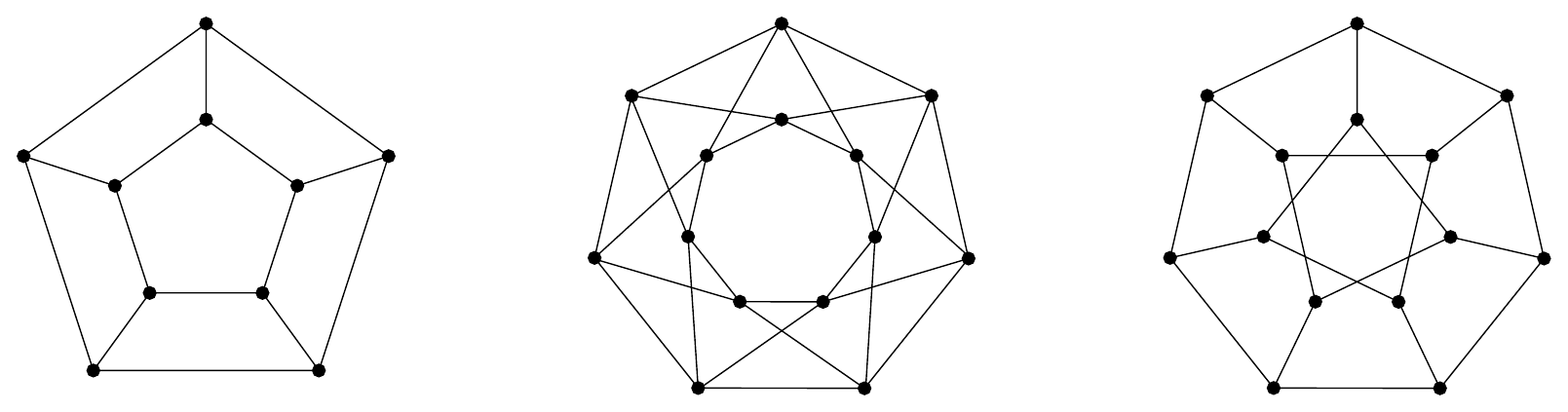}   
    \caption{
    Graphs whose independent domination number achieves a tight upper bound}
    \label{fig:tight}
\end{figure}

Recall that both $K_{3, 3}$ and $C_5\square K_2$ must be excluded, which is done by the restriction on the $4$-cycle.
Moreover, in \cite{goddard2012independent}, there are infinitely many connected cubic graphs $G$ satisfying $i(G) = \frac{3}{8}|V(G)|$, which all have $4$-cycles. 

We actually prove a stronger statement, which applies to subcubic graphs.
Note that Theorem~\ref{mainthm:cubic} is a direct consequence of the following theorem:

\begin{theorem}\label{mainthm:subcubic}
If $G$ is a subcubic graph  without $4$-cycles, then $14i(G) \le 14n_0(G)+9n_1(G)+6n_2(G)+5n_3(G)$, where $n_i(G)$ denotes the number of vertices of degree $i$ in $G$.
\end{theorem}

There are infinitely many connected graphs that are tight for the above result. 
Some simple ones are the edgeless graph, the $7$-cycle, and the right graph in Figure~\ref{fig:tight}. 
We now construct an infinite family of connected graphs that are tightness examples.
For non-negative integers $k$ and $\ell$ satisfying $k+\ell\ge 5$, let $T_{k,\ell}$ be the graph obtained from a $(k+4\ell)$-cycle $x_1\ldots x_{k+4\ell}x_1$ by adding a pendent vertex to each of $x_1, \ldots, x_k$, and adding a vertex joining $x_{k+4i+1}$ and $x_{k+4i+4}$ for each 
$i\in\{0,\ldots, \ell-1\}$ 
(see Figure~\ref{fig:coefficient}).
Then $n_0(T_{k,\ell}) = 0$, $n_1(T_{k,\ell}) = k$, $n_2(T_{k,\ell}) = 3\ell$, $n_3(T_{k,\ell}) = k+2\ell$, and $i(T_{k,\ell})=k+2\ell$.
Hence, equality in Theorem~\ref{mainthm:subcubic} holds.

\begin{figure}
    \centering
    \includegraphics[scale=0.8, page=2]{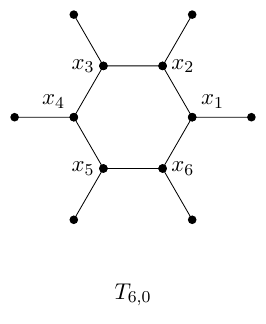}     \qquad \qquad
    \includegraphics[scale=0.8, page=3]{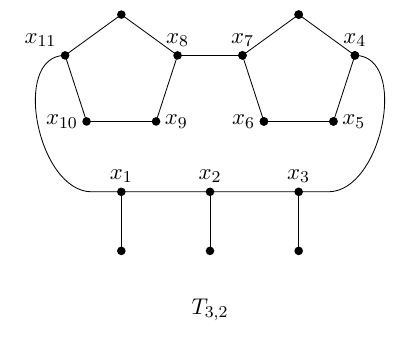}
    \caption{The graphs $T_{6,0}$ and $T_{3,2}$}
    \label{fig:coefficient}
\end{figure}

\smallskip
Disproving a conjecture by Henning and Southey~\cite{southey2013domination}, O and West~\cite{suil2016cubic} constructed an infinite family of  $2$-connected cubic graphs $G$ satisfying $\frac{i(G)}{\gamma(G)}=\frac{5}{4}$. 
They asked if this is best possible. 

\begin{question}[\cite{suil2016cubic}]
Does $\frac{i(G)}{\gamma(G)} \le \frac{5}{4}$ hold for a connected cubic graph $G$ on sufficiently many vertices?
\end{question}

We answer this question in the affirmative for cubic graphs without $4$-cycles.

\begin{theorem}\label{thm:ratio}
If $G$ is a cubic graph without $4$-cycles, then $\frac{i(G)}{\gamma(G)} \le \frac{5}{4}$.
\end{theorem}

As all $2$-connected cubic graphs $G$ satisfying $\frac{i(G)}{\gamma(G)} = \frac{5}{4}$ provided in \cite{suil2016cubic}  have a $4$-cycle, we raise  the following question:

\begin{question}
Is there a cubic graph $G$ without $4$-cycles satisfying $\frac{i(G)}{\gamma(G)} = \frac{5}{4}$?
\end{question}

We prove Theorem~\ref{mainthm:subcubic} and Theorem~\ref{thm:ratio} in Section~\ref{sec:main} and Section~\ref{sec:ratio}, respectively.
We end this section with some definitions frequently used.

Given a graph $G$ and $X \subseteq V(G)$, let $G[X]$ denote the subgraph of $G$ induced by $X$, and let $G-X$  denote the subgraph of $G$ induced by $V(G) \setminus X$.  
For $X \subseteq V(G)$, let $\partial_G(X)$ be the set of edges whose one end is in $X$ and the other end is in $G-X$.
For a vertex $v \in V(G)$, let $d_G(v)$ denote the {\it degree} of $v$ in $G$.
For a non-negative integer $k$, a vertex $v$ is a {\it $k$-vertex} if $d_G(v)=k$, and a {\it $k$-neighbor} of $v$ is a neighbor of $v$ that is also a $k$-vertex. A {\it $2$-step neighbor} of $v$ is a vertex that can reach $v$ by a path of length $2$. 
A {\it $2$-step $k$-neighbor} of $v$ is a $2$-step neighbor that is a $k$-vertex. 
For a non-negative integer $k$, let $n_k(G)$ denote the number of $k$-vertices of $G$.
For each vertex $v \in V(G)$, let $N_G(v)$ denote the set of neighbors of $v$, and let 
$N_G[v]=N_G(v)\cup \{v\}$.
For  $X \subseteq V(G)$, let $N_G(X) = \bigcup_{v \in X} N_G(v)$ and $N_G[X] = N_G(X) \cup X$.
For $v \in V(G)$ and $X \subseteq V(G)$, an {\it $X$-vertex} is a vertex in $X$, and an {\it $X$-neighbor} of $v$ is a neighbor of $v$ that is in $X$. 
Throughout the paper, if it is clear which graph we are referring to, then we will simply use $d(v)$, $N(v)$, $N[v]$, $N[X]$, and $\partial(X)$,
instead of $d_G(v)$, $N_G(v)$, $N_G[v]$, $N_G[X]$, and $\partial_G(X)$, respectively. 
For $\{v_1, \ldots, v_k\}\subseteq V(G)$, we simply write $N[v_1,  \ldots, v_k]$ instead of $N[\{v_1,\ldots,v_k\}]$.

\section{Ratio of independent domination and number of vertices}\label{sec:main}

Given a vertex $v$ of a subcubic graph $G$, its {\it weight} $w_G(v)$ is defined as follows:
\begin{eqnarray*}
w_G(v) = \begin{cases}14 &\text{ if } d(v) = 0 \\
9 &\text{ if } d(v) = 1\\
6 &\text{ if } d(v) = 2\\
5 &\text{ if } d(v) = 3\\
\end{cases}
\end{eqnarray*}
The {\it weight} of a non-empty set $X\subseteq V(G)$, denoted $w_G(X)$, is defined by $w_G(X) = \displaystyle\sum_{v \in X} w_G(v)$, and 
the {\it weight} of the entire graph $G$, denoted $w(G)$, is the weight of $V(G)$.
That is,
\[w(G) = w(V(G))=\displaystyle\sum_{v \in V(G)}w_G(v) = 14n_0(G) + 9n_1(G) + 6n_2(G) +5n_3(G).\]
For a non-empty subset $X$ of $V(G)$,
we denote the sum of weight changes of the vertices in $G-X$ by $c_G(X)$.
That is,
\begin{eqnarray*}
c_G(X) = \sum_{v \in V(G-X)} (w_{G-X}(v)-w_G(v)) =  w(G-X) -( w(G) - w_G(X)).
\end{eqnarray*}

To show Theorem~\ref{mainthm:subcubic}, suppose to the contrary that there exists a counterexample, and let $G$ be a counterexample with the minimum number of vertices and edges. 
This implies that 
$G$ is a subcubic graph without $4$-cycles satisfying $14i(G) > w(G)$, and every proper subgraph $H$ of $G$ satisfies $14i(H) \le w(H)$.

If $G$ is the disjoint union of two graphs $G_1$ and $G_2$, then by the minimality of $G$,  
\begin{eqnarray*}
&& 14i(G) = 14i(G_1)+14i(G_2) \le w(G_1)+w(G_2) = w(G),
\end{eqnarray*}
which is a contradiction.
Thus, $G$ is connected. 
One can check that the theorem holds for all graphs with at most four vertices, and all paths and cycles on at least five vertices. 
Thus, $G$ is a graph on at least five vertices and has a $3$-vertex.

We will often use the following observation. 
For a vertex $v$, a vertex of $G-N[v]$ cannot be adjacent to two vertices in $N[v]$ since $G$ has no $4$-cycles. 
By the same reason, two vertices cannot have two common neighbors. 
We now prove a key lemma that will be crucial in our future arguments. 

\begin{keylemma}\label{fact}
If $S$ is an independent set of $G$, then $14|S| > w_G(N[S])-c_G(N[S])$.
\end{keylemma}

\begin{proof}
 Let $U=N[S]$, and let $S'$ be a minimum independent dominating set of $G-U$. By the minimality of $G$, we have $w(G-U) \ge 14i(G-U) = 14 |S'|$.
Since $S$ is an independent dominating set of $G[U]$ and $S\cup S'$ is an independent set, $S\cup S'$ is an independent dominating set of $G$. 
Since
\begin{eqnarray*}
14|S| + 14|S'| \ge 14i(G) > w(G) = w(G-U)+w_G(U)-c_G(U) \ge 14|S'|+w_G(U)-c_G(U), 
\end{eqnarray*}
 we obtain the desired conclusion.
\end{proof}

Our first goal is to prove that $G$ has no $1$-vetices. 
We first lay out some structural observations.

\begin{claim}\label{clm:vcond}
The following do not appear in the graph $G$:
\begin{enumerate}[\rm(i)]
\item A vertex with two $1$-neighbors.
\item A vertex with only $2^-$-neighbors.
\item A $3$-vertex with a $1$-neighbor and a $2$-neighbor.
\item A triangle $vuwv$ such that $v$ has a $1$-neighbor.
\item A triangle with a $2$-vertex. 
\item A path $vuw$ such that $v,u,w$ are $3$-vertices, $v$ has a $1$-neighbor and $u$ has no $1$-neighbor.
\end{enumerate}
\end{claim}

\begin{proof}
 In the following, let $v$ be a vertex, and let $t_i$ be the number of $i$-neighbors of $v$, so $t_1+t_2+t_3=d(v)$. Let $U=N[v]$, so $w_G(U)=9t_1+6t_2+5t_3+w_G(v)$.

(i): Suppose to the contrary that $t_1\ge 2$. 
Then $w_G(U)\ge 9 \cdot 2 + \min\{6, 5+5, 5+6\} =24$ and $|\partial(U)|\le 2$.
If $G-U$ has two isolated vertices, then $G$ is a tree on $6$ vertices, so $14i(G) = 14 \cdot 3 =42 \le 46 = 9 \cdot 4 + 5 \cdot 2 =w(G)$, which is a contradiction.
Thus $G-U$ has at most one isolated vertex, so $c_G(U)\le 5+3 =8$.
By the {\kl}, $14>24-8$, which is a contradiction.
Hence (i) holds. 

Actually, by the same logic, we know that
$c_G(U)\le 5t_2+ (5+3)t_3=5t_2+8t_3$.
Note that every $3$-neighbor of $v$ has at most one $1$-neighbor by $(i)$. 
By the {\kl},
\[14 > 9t_1+6t_2+5t_3 + w_G(v) - (5t_2+ 8t_3) = 9t_1+t_2-3t_3 + w_G(v).\]

(ii):  Suppose to the contrary that $t_3=0$, and assume $v$ is such a vertex with minimum degree.
If $t_1\ge 1$, then by the {\kl}, $14>9t_1+t_2+w_G(v)\ge 9 + 5$, which is a contradiction, so $t_1=0$.
Therefore, all neighbors of $v$ are $2$-vertices. 

Suppose that $v$ is a $1$-vertex, and let $u$ be the neighbor of $v$. 
Since $G$ is a graph on at least $5$ vertices,
the neighbor $w$ of $u$ that is not $v$ must be a $2^+$-vertex.  
Since $w_G(U) = 9+6=15$, and $c_G(U)\in\{1,3\}$, the {\kl} implies $c_G(U)= 3$, so $w$ must be a $2$-vertex. 
Now, we know $w_G(N[u])=9+6\cdot 2=21$ and $c_G(N[u])\le 5$.
The {\kl} implies $14>21-5$, which is a contradiction.
In particular, there is no $1$-vertex with a $2^-$-neighbor. 

Now suppose that $d(v)\ge 2$.
Since there is no $1$-vertex with a $2^-$-neighbor, 
the $2$-step neighbors of $v$ are  $2^+$-vertices. 
 
If $v$ is a $2$-vertex, then $w_G(U) = 6 \cdot 3 = 18$.
Since $G$ is not a cycle, we can take $v$ to be a vertex such that at least one $2$-step neighbor of $v$ is a $3$-vertex. 
Since $G$ has no $4$-cycle, $c_G(U) \le 3+1 =4$, so the {\kl} implies $14>18-4 $, which is a contradiction.
If $v$ is a $3$-vertex, then $w_G(U)=6 \cdot 3 +5=23$ and $c_G(U)\le 3 \cdot 3= 9$, so the {\kl} implies $14>23-9 $, which is a contradiction.
Hence (ii) holds.

(iii): Suppose to the contrary that $v$ is a $3$-vertex, $t_1\ge 1$, and $t_2\ge 1$.
Then $t_1=t_2=t_3=1$ by (ii), and therefore  $w_G(U)=9 +6+5\cdot 2=25$. 
By (ii), the $2$-neighbor of $v$ has no $1$-neighbor and by (i), the $3$-neighbor of $v$ has at most one $1$-neighbor, so $c_G(U)\le 5+ 3 \cdot 2=11$. 
The {\kl} implies $14>25-11$, which is a contradiction.
Hence (iii) holds.

(iv):  Suppose to the contrary that there is a triangle $vuwv$ such that $v$ has a $1$-neighbor. Then $w_G(U)\ge 9+5\cdot 3=24$ and $c_G(U)\le 5\cdot 2=10$. The {\kl} implies $14>24-10$, which is a contradiction.
Hence (iv) holds.

(v): Suppose to the contrary that there is a triangle $vuwv$ such that $v$ is a $2$-vertex.
If $u$ is a $2$-vertex, then $w$ must be a $3$-vertex since $G$ is a graph on at least $5$ vertices. 
By (ii) and (iv), the neighbor of $w$ not on $vuwv$ must be a $3$-vertex, so $w_G(U) \ge 6 \cdot 2 + 5 = 17$ and $c_G(U) \le 1$.
The {\kl} implies $14>17-1$, which is a contradiction.
Thus, $u$ is a $3$-vertex, and by symmetry, $w$ is also a $3$-vertex, so $w_G(U) = 6 + 5 \cdot 2 = 16$.
Let $u',w'$ be the neighbor of $u,w$, respectively, that is not on the triangle $vuwv$. 
Note that $u'\neq w'$ since $G$ has no $4$-cycles.
By (iv), both $u'$ and $w'$ are $2^+$-vertices.
If $u', w'$ are $3$-vertices, then $c_G(U) \le 1 \cdot 2 = 2$,  so the {\kl} implies $14>16-2$, which is a contradiction.
Without loss of generality, let $u'$ be a $2$-vertex.
Note that $u'$ has no $1$-neighbor by (ii).
Now, $w_G(N[u]) \ge 6 \cdot 2 + 5 \cdot 2 = 22$ and $c_G(N[u]) \le 3 \cdot 2 = 6$. The {\kl} implies $14>22-6$, which is a contradiction.
Hence (v) holds. 

(vi): Suppose to the contrary that there is a path $vuw$ such that $v, u, w$ are $3$-vertices, $v$ has a $1$-neighbor, and $u$ has no $1$-neighbor.
By (i) and (iii), $v$ has a $3$-neighbor $x$ other than $u$. Note that $x\neq w$ by (iv). 
By (iii), $c_G(U) \le 1+3+\max\{5 + 1, 3+3\} = 10$. 
Since $w_G(U)\ge 9 + 5 \cdot 3 = 24$, the {\kl} implies  $14>24-10$, which is a contradiction.
Hence (vi) holds. 
\end{proof}

\begin{claim}\label{claim:n1}
The graph $G$ has no $1$-vertex. 
\end{claim}   
\begin{proof}                                  
Suppose to the contrary that $G$ has a $1$-vertex. 
Let $X$ be the set of $3$-vertices with $1$-neighbors.
Note that $G[X]$ has maximum degree at most $2$ and $X$ is not empty since every $1$-vertex is  adjacent to a $3$-vertex by Claim~\ref{clm:vcond}~(ii).
Take a component $P$ of $G[X]$.
If $P$ is a cycle, then $G$ is isomorphic to $T_{k,0}$ for some $k \ge 3$ (see~Figure~\ref{fig:coefficient}), which is a contradiction to our assumption that $14i(G) > w(G)$. Thus $P$ is a path, and let $u_1$ be an end vertex of $P$. 
Let $N(u_1)=\{v_1,u_2,x\}$, where
$v_1$ is the $1$-neighbor of $u_1$ and  $x\not\in X$.
By Claim~\ref{clm:vcond}~(i),~(iii), and~(iv), $x$ and $u_2$ are $3$-vertices and $xu_2\not\in E(G)$.
Let $N(x)=\{u_1,x_1,x_2\}$.
We know $x_1$ and $x_2$ are $2$-vertices by Claim~\ref{clm:vcond}~(vi).

\smallskip

\begin{figure}
    \centering
    \includegraphics[width=18cm, page=4]{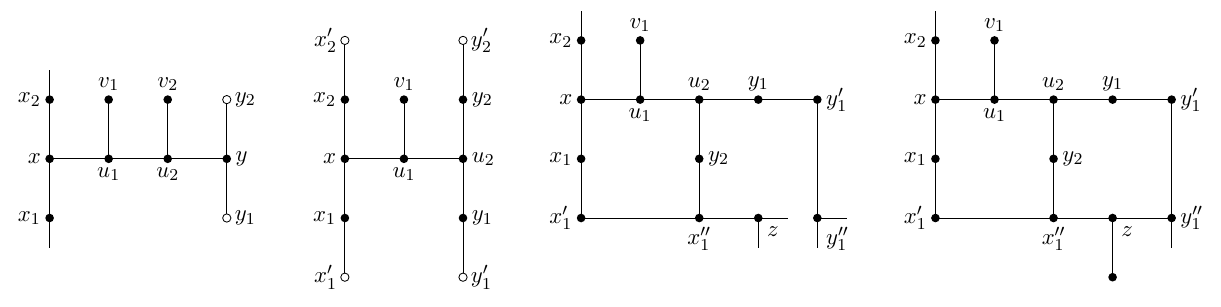}
    \caption{Illustrations for the proof of Claim~\ref{claim:n1}}
    \label{fig:n1}
\end{figure}

\noindent{\bf(Case 1)} Suppose that $u_2$ is on $P$. 
Let $N(u_2)=\{v_2,y,u_1\}$, where $v_2$ is the $1$-neighbor of $u_2$. 
By Claim~\ref{clm:vcond}~(i),~(iii), and~(iv), $y$ is a $3$-vertex distinct from $x$.
Let $N(y)=\{u_2,y_1,y_2\}$.
Note that $xy \notin E(G)$ since $G$ has no $4$-cycles.
See the first figure of Figure~\ref{fig:n1} for an illustration.
Suppose that $u_1$ is not on a $6$-cycle.
Consider $G' = (G-\{u_1,u_2,v_1,v_2\})+xy$.
Since $G'$ is a subcubic graph without $4$-cycles, by the minimality of $G$, $14i(G') \le w(G')$.
We can always add two vertices in $\{u_1,u_2,v_1,v_2\}$ to a minimum independent dominating set of $G'$ to obtain an independent dominating set of $G$, so $i(G) \le i(G') + 2$.
Since $w(G) = w(G') + 28$, we obtain $14i(G) \le 14i(G') + 28 \le w(G') + 28 = w(G)$, which is a contradiction.
Thus $u_1$ must be on a $6$-cycle, say $xu_1u_2yy_1x_1x$. 

Suppose $y \in X$, and assume $y_2$ is a $1$-neighbor of $y$. 
Then, $y_1$ is a $3$-vertex by Claim~\ref{clm:vcond}~(i) and~(iii).
Since $x_1$ is a 2-neighbor of $y_1$, Claim~\ref{clm:vcond}~(iii) implies $y_1 \notin X$.
Now, $w_G(N[x,v_1,v_2,y]) = 9\cdot 3+ 6 \cdot 2 + 5\cdot5 = 64$ and $c_G(N[x,v_1,v_2,y]) \le \max\{3 +3, 8\} = 8$, so the {\kl} implies $56=4\cdot 14 > 64-8 = 56$, which is a contradiction.
Thus $y \notin X$. 
By Claim~\ref{clm:vcond}~(vi), $y_1$ and $y_2$ are $2$-vertices. 
Now, $w_G(N[x_1,v_1,u_2]) = 9 \cdot 2 + 6 \cdot 2 + 5 \cdot 4 = 50$ and $c_G(N[x_1,v_1,u_2]) \le \max\{3+3,8\}=8$, so the {\kl} implies $42=3 \cdot 14 > 50-8=42$, which is a contradiction.

\smallskip

\noindent{\bf(Case 2)} 
Suppose that $u_2$ is not on $P$. 
Let $N(u_2)=\{u_1,y_1,y_2\}$.
By Claim~\ref{clm:vcond}~(vi),  $y_1$ and $y_2$ are $2$-vertices. 
For $i\in\{1, 2\}$, let $x'_i$ and $y'_i$ be the neighbor of $x_i$ and $y_i$ distinct from $x$ and $u_2$, respectively. 
See the second figure of Figure~\ref{fig:n1} for an illustration.
Since $G$ has no $4$-cycles, $x'_1,x'_2,u_2$ are all distinct.
If $x'_1$ and $x'_2$ are $3$-vertices, then $w_G(N[x,v_1]) = 9 + 6 \cdot 2 + 5 \cdot 2 = 31$ and $c_G(N[x,v_1]) \le 1 \cdot 3 = 3$, so the {\kl} implies $28=2\cdot 14 >31-3=28$, which is a contradiction.
Without loss of generality, we may assume $x'_1$ is a $2$-vertex. 
By symmetry, we may assume $y'_1$ is a $2$-vertex. 

\smallskip

\noindent\textbf{(Case 2-1)} Suppose that $x'_1x'_2 \in E(G)$ and $y'_1y'_2\in E(G)$.
Note that $x,u_1,u_2$ cannot all be on the same $5$-cycle.
Thus $G' = (G -\{u_1,v_1\})+xu_2$ has no $4$-cycles.
By the minimality of $G$, $14i(G') \le w(G')$.
Let $S'$ be a minimum independent dominating set of $G'$.
We will define an independent dominating set $S$ of $G$ of size $|S'|+1$. 
If $S'$ does not contain a vertex in $\{x,u_2\}$, then let $S=S' \cup \{v_1\}$.
If $x\in S'$, then exactly one of $x'_1, x'_2$ is in $S'$, so when $x'_1\in S'$ (resp. $x'_2\in S'$),  let $S=(S'\cup\{x_2, u_1\})\setminus\{x\}$ (resp. $S=(S'\cup\{x_1, u_1\})\setminus\{x\}$).
By symmetry, the case when $u_2\in S$ is also resolved. 
In every case, $i(G) \le i(G') +1$.
Since $w(G) = w(G') + 14$, we have $14i(G) \le 14i(G')+14 \le w(G') + 14 = w(G)$, which is a contradiction.

\smallskip

\noindent \textbf{(Case 2-2)} 
Without loss of generality, suppose that $x'_1x'_2 \not\in E(G)$.
Let $x''_1$ be the neighbor of $x'_1$ distinct from $x_1$.
By Claim~\ref{clm:vcond}~(v), we know $x''_1 \neq x$.
By Claim~\ref{clm:vcond}~(ii), $x''_1$ is a $3$-vertex with a $3$-neighbor. 
If $x''_1 = u_2$, then $w_G(N[u_1, x'_1]) = 9 + 6 \cdot 2 + 5 \cdot 3 = 36$ and $c_G(N[u_1, x'_1]) \le 3 \cdot 2 = 6$, so the {\kl} implies $28=2 \cdot 14 > 36-6=30$, which is a contradiction.
Thus $x''_1 \neq u_2$, and 
$w_G(N[u_1, x'_1]) = 9 + 6 \cdot 2 + 5 \cdot 4 = 41$.
If $G-N[u_1, x'_1]$ has no isolated vertices, then $c_G(N[u_1, x'_1]) \le 3 \cdot 4 + 1 = 13$, so the {\kl} implies $28=2 \cdot 14 > 41-13=28$, which is a contradiction.
Thus, $G-N[u_1, x'_1]$ must have an isolated vertex.
Since $y'_1$ is a $2$-vertex and $x'_1x'_2 \not\in E(G)$, we know $x''_1y_2 \in E(G)$.

Let $z$ be the neighbor of $x''_1$ distinct from $x'_1$ and $y_2$, so $z$ is a $3$-vertex. 
Let $y''_1$ be the neighbor of $y'_1$ other than $y_1$, so by Claim~\ref{clm:vcond}~(ii), $y''_1$ is a $3$-vertex. 
See the third figure of Figure~\ref{fig:n1} for an illustration.
Since $w_G(N[x''_1]) = 6 \cdot 2 + 5 \cdot 2 = 22$, the {\kl} implies $c_G(N[x''_1]) \ge 9$. 
Thus, together with Claim~\ref{clm:vcond}~(i) and~(iii), we know that the neighbors of $z$ other than $x''_1$ are either both $2$-vertices or a $1$-vertex and a $3$-vertex.
Since $w_G(N[x''_1,y_1])=6 \cdot 4 + 5 \cdot 3 = 39$, by the {\kl}, $28 = 2 \cdot 14 > 39 - c_G(N[x''_1,y_1])$, so  $c_G(N[x''_1,y_1]) \geq 12$.
If $z = y''_1$ or $zy''_1 \not\in E(G)$, then $c_G(N[x''_1,y_1]) \le 3+1\cdot2+\max\{3+3,1+5\}=11$, which is a contradiction.
Thus $zy''_1\in E(G)$, and moreover, $z$ has a $1$-neighbor.
See the last figure of Figure~\ref{fig:n1} for an illustration.
 
Now $w_G(N[u_1,x_1,y_2,y'_1,z]) = 9 \cdot 2 + 6 \cdot 5 + 5 \cdot 6 = 78$.
Note that $y''_1$ has no $1$-neighbor by Claim~\ref{clm:vcond}~(iii).
So $c_G(N[u_1,x_1,y_2,y'_1,z]) \le \max\{3+3,8\} = 8$.
By the {\kl}, $70 = 5 \cdot 14 > 78 - 8$, which is a contradiction.
\end{proof}

We will now analyze the structure between $2$-vertices and $3$-vertices. 
In the following, let $A_i$ be the set of $3$-vertices with exactly $i$ $2$-neighbors, and let $B_i$ be the set of $2$-vertices with exactly $i$ $2$-neighbors. 
By Claim~\ref{clm:vcond}~(ii),  $A_3=B_2=\emptyset$, so $V(G)$ can be partitioned into $A_0, A_1, A_2, B_0$, and $B_1$.
Note that $B_0$ is an independent set in $G$ and $G[B_1]$ is the disjoint union of $K_2$.

\begin{claim}\label{clm:one3}
The following holds in the graph $G$:
\begin{enumerate}[\rm(i)]
    \item Each $B_1$-vertex $v$ has at least one $2$-step $2$-neighbor.
    In particular, $v$ has a $B_1$-neighbor and an $A_2$-neighbor.
    \item Each $A_i$-vertex $v$ has at least $i+1$ $2$-step $2$-neighbors.
    \begin{enumerate}[\rm({ii})-1]
        \item If $i=2$, then the three neighbors of $v$ are in $A_2, B_1, B_1$, or $A_1, B_1, B_1$, or $A_2, B_0, B_1$. 
        \item If $i=0$, then $v$ has an $A_1$-neighbor, which has an $A_2$-neighbor.
    \end{enumerate}

\end{enumerate}
\end{claim}
\begin{proof}
(i): Let $v\in B_1$, so $v$ has a $2$-neighbor $v_1$ and a $3$-neighbor $v_2$.
If $v$ has no $2$-step $2$-neighbors,
then $w_G(N[v]) = 6 \cdot 2 + 5 = 17$ and $c_G(N[v]) \le 1 \cdot 3 = 3$, so
the  {\kl} implies $14>17-3$, which is a contradiction.
By Claim~\ref{clm:vcond}~(ii), $v_1\in B_1$. 
Since $v$ must have a $2$-step $2$-neighbor and $v_2$ cannot have only $2$-neighbors by Claim~\ref{clm:vcond}~(ii), we know $v_2\in A_2$.

(ii): Suppose to the contrary that a vertex $v \in A_i$  has at most $i$ $2$-step $2$-neighbors. 
Then $w_G(N[v]) = 6 \cdot i + 5 \cdot (4-i) = i+20$ and $c_G(N[v]) \le 3\cdot i + 1 \cdot ((3-i)\cdot2)=i+6$.
The {\kl} implies $14>(i+20)-(i+6)=14$, which is a contradiction.

If $i=2$, then the three neighbors of $v$ are in $A_2, B_1, B_1$, or $A_1, B_1, B_1$, or $A_2, B_0, B_1$. 

If $i=0$, then $v$ is an $A_0$-vertex. 
Since an $A_2$-vertex cannot have an $A_0$-neighbor, $v$ has no $A_2$-neighbors.
Since $v$ must have a $2$-step $2$-neighbor $w$, the common neighbor $u$ of $w$ and $v$ must be an $A_1$-vertex. 
Since a $B_1$-vertex cannot have an $A_1$-neighbor by (i), $w$ must be a $B_0$-vertex.
Since $u$ must have two $2$-step $2$-neighbors, the $3$-neighbor of $u$ distinct from $v$ is an $A_2$-vertex. 
\end{proof}

\begin{claim}\label{clm:a2b1}
The following holds in the graph $G$: 
\begin{enumerate}[\rm(i)]
\item The set $A_2$ is not empty. 
\item There is a cycle $C: x_1 x_2 \ldots x_k x_1$ in $G[A_2 \cup B_1]$ such that if $x_i,x_{i+1} \in  A_2$, then $x_i$ has a $B_0$-neighbor.
In particular, each $A_2$-vertex on $C$ has at most one $A_2$-neighbor on $C$ and each $B_1$-vertex on $C$ has exactly one $B_1$-neighbor on $C$. 
\item If two $A_2$-vertices $x,y$ have a common $B_0$-neighbor $v$, then there is a $5$-cycle $vxx_1y_1yv$ such that $x_1,y_1$ are $B_1$-vertices.
\end{enumerate}
\end{claim}

\begin{proof}
(i): Suppose to the contrary that $G$ has no $A_2$-vertex.
Claim~\ref{clm:one3} implies that there are neither $B_1$-vertices nor $A_0$-vertices. 
Thus, every $3$-vertex has a $2$-neighbor, so the set of $2$-vertices is a dominating set. 
Since every $2$-vertex is a $B_0$-vertex, $B_0$ dominates $G$.
Moreover, $B_0$ is an independent dominating set in $G$ by definition, so $i(G)\leq n_2(G)$.
Also, $n_3(G) = 2n_2(G)$ since there are no $A_2$-vertices. 
Thus, $w(G)=6n_2(G)+5n_3(G)=16n_2(G)\geq 16i(G) > 14i(G)$, which is a contradiction.
Hence, $G$ must have some $A_2$-vertices. 

(ii): Note that there exists an $A_2$-vertex $v_1$ by (i). 
We find a cycle $C:x_1x_2\ldots x_1$ of $G[A_2\cup B_1]$ containing $v_1$ by the following algorithm:
Let $x_1=v_1$. 
Consider an $A_2$-vertex $x_i$ on $C$. 
(The indices of the vertices on $C$ are considered modulo $k$.)
If $x_i$ has an $B_1$-neighbor $u_1$ that is not $x_{i-1}$, then there is a path $u_1u_2u_3$ where $u_1,u_2$ are $B_1$-vertices and $u_3$ is an $A_2$-vertex by Claim~\ref{clm:one3}~(i) and~(ii).
Let $x_{i+1}=u_1$, $x_{i+2}=u_2$, and $x_{i+3}=u_3$. 
(Note that $x_ix_{i+1}x_{i+2}x_{i+3}$ is indeed a path since $x_i \neq x_{i+3}$ by Claim~\ref{clm:vcond}~(v).)
If the only $B_1$-neighbor of $x_i$ is $x_{i-1}$, then let $x_{i+1}$ be the $A_2$-neighbor of $x_i$, which exists by Claim~\ref{clm:one3}~(ii).  

Since $G$ is a finite graph, this algorithm must terminate. 
Note that each $A_2$-vertex $x_i$ with $x_{i+1}\in A_2$ must have an $B_0$-neighbor by Claim~\ref{clm:one3}~(ii).
In particular, by the algorithm, each $A_2$-vertex on $C$ has at most one $A_2$-neighbor on $C$ and each $B_1$-vertex on $C$ has exactly one $B_1$-neighbor on $C$. 

(iii): 
By Claim~\ref{clm:one3}~(ii), $x$ and $y$ have a $B_1$-neighbor $x_1$ and $y_1$, respectively.
Suppose to the contrary that $x_1y_1$ is not an edge.
By Claim~\ref{clm:one3}~(i), $x_1$ has a $2$-neighbor, so $w_G(N[x_1,y]) = 6 \cdot 4 + 5 \cdot 3 = 39$.
By Claim~\ref{clm:vcond}~(ii), every $2$-step neighbor of $x_1$ not in $N[x_1,y]$ is a $3$-vertex. 
By Claim~\ref{clm:one3}~(ii), every $2$-step neighbor of $y$ not in $N[x_1,y]$ is a $2$-vertex.
Thus, $c_G(N[x_1,y]) \le 3 \cdot 3 + 1 \cdot 2 = 11$, so the {\kl} implies $28 = 14 \cdot 2 > 39-11=28$, which is a contradiction.
Hence, $x_1y_1\in E(G)$, so $vxx_1y_1yv$ is a $5$-cycle.
\end{proof}

By Claim~\ref{clm:a2b1}~(ii), there is a cycle $C: x_1 x_2 \ldots x_k x_1$ in $G[A_2 \cup B_1]$ such that if $x_i,x_{i+1} \in A_2$, then $x_i$ has a $B_0$-neighbor.
(The indices of the vertices on $C$ will always be considered modulo $k$.)  
Define the following subset of $V(C)$:  \begin{eqnarray*}
X  &=&\{x_i \in B_1  \mid x_{i-1} \in A_2, x_{i+1} \in B_1\}.\end{eqnarray*}
Let $y_i$ be the $B_0$-neighbor of $x_i$ when $x_i,x_{i+1} \in A_2$, and let $Y$ be the set of all $y_i$'s. 
Note that $X\cup Y$ consists of $2$-vertices, and $X\cup Y$ is an independent set in $G$ by definition.
Let $Y_j$ be the set of $Y$-vertices with  exactly $j$ $A_2$-neighbors.
If $y_i$ is a $Y_2$-vertex, then $y_ix_ix_{i-1}x_{i-2}x_{i-3}y_i$ is a $5$-cycle by Claim~\ref{clm:a2b1}~(iii).
Thus, both neighbors of a $Y_2$-vertex are $(A_2\cap V(C))$-vertices, and one of them has an $X$-neighbor.
See Figure~\ref{fig:overall} for an illustration.

\begin{figure}
    \centering 
    \includegraphics[page=5]{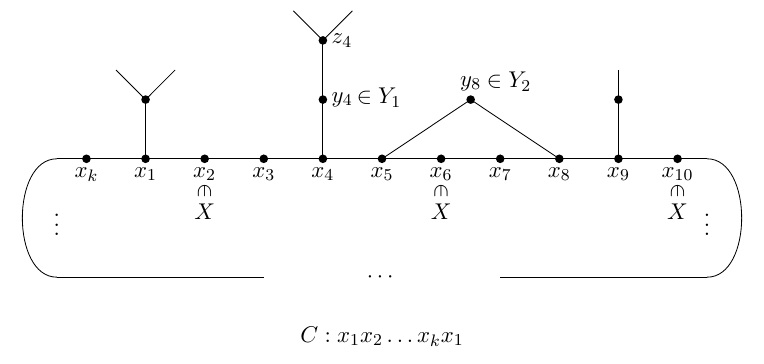}
    \caption{An illustration for the cycle $C$}
    \label{fig:overall}
\end{figure}

Define the following two sets:
\begin{eqnarray*}
X^* = \{ x_{i-1}  \mid x_{i}\in X\} \setminus N(Y),
\quad
Y^* = \{z_i\mid\mbox{$z_i$ is the $A_1$-neighbor of $y_i$ distinct from $x_i$}\}.
\end{eqnarray*}
Note that $z_i$ is an $A_1$-vertex and is not on $C$, and $|X^*|=|X|-|Y_2|$ and $|Y^*|=|Y_1|$,

\begin{claim}\label{clm:two:C}
If $X^*$-vertices $x_i,x_j$ have a common neighbor $v\not\in  V(C)$, then $v$ is a $3$-vertex.  
\end{claim}
\begin{proof}
Suppose to the contrary that $v$ is a $2$-vertex. 
Since $x_i, x_j$ are $A_2$-vertices, we know $v \in B_0$.
By Claim~\ref{clm:a2b1}~(iii), there is a $5$-cycle $x_ix_{i+1}x_{i+2}x_jvx_i$, where $x_{i+1}, x_{i+2}$ are $B_1$-vertices, so $j=i+3$.
Since $x_{i+3}$ is an $X^*$-vertex, we obtain that $x_{i+4}$ is a $B_1$-vertex, which is a contradiction since now $x_{i+3}$ has three $2$-neighbors $x_{i+2}$, $x_{i+4}$, and $v$.
\end{proof}

\begin{claim}\label{clm:z2}
Two $Y^*$-vertices do not have a common neighbor.
\end{claim}
\begin{proof}
Suppose to the contrary that two distinct $Y^*$-vertices $z_i$ and $z_j$ have a common neighbor $v$. 
Since each $Y^*$-vertex is an $A_1$-vertex, $v$ must be a $3$-vertex. 
Note that $x_{i-1}, x_{i-2} \in B_1$ and $x_{i+1} \in A_2$, which further implies $x_{i+1}$ is a $2$-vertex in $G-N[x_{i-1},z_i]$.

Suppose $v \in A_0$. 
Then $w_G(N[x_{i-1}, z_i]) = 6 \cdot 3 + 5 \cdot 4 = 38$, and the neighbor of $z_i$ distinct from $v$ and $y_i$ is an $A_2$-vertex by Claim~\ref{clm:one3}~(ii)-2.
Thus if there is a common $2$-step neighbor of $x_{i-1}$ and $z_i$, then $v$ is adjacent to either $x_{i-3}$ or $x_{i+1}$, which is impossible since $v \in A_0$ and $x_{i-3},x_{i+1} \in A_2$.
Thus $c_G(N[x_{i-1},z_i]) \le 3 \cdot 2 + 1 \cdot 4 = 10$, so the {\kl} implies $28>38-10$, which is a contradiction.
Thus $v \in A_1$, so $v$ has a $B_0$-neighbor by Claim~\ref{clm:one3}~(i).

Let $U_1 = N[v, x_i, x_j]$, so $w_G(U_1) = 6 \cdot 5 + 5 \cdot 7 = 65$.
Note that every $2$-step neighbor of $v$ not in $U_1$ is a $3$-vertex, and every $2$-step neighbor of $x_i$ or $x_j$ not in $U_1$ is a $2$-vertex.
By Claim~\ref{clm:two:C}, if $x_{i+1}, x_{j+1} \in X^*$ have a common neighbor $u \not\in  V(C)$, then $u$ is a $3$-vertex, which is impossible since $x_i$ and $x_j$ are the $3$-neighbors of $x_{i+1}$ and $x_{j+1}$, respectively.
Thus, there are at most two vertices in $G-U_1$ that are incident with two edges of $\partial(U_1)$, namely, $x_{i-2}$ and $x_{j-2}$.
If there is at most one vertex in $G-U_1$ that is incident with two edges of $\partial(U_1)$, then $c_G(U_1) \le 8 + 3 \cdot 4 + 1 \cdot 3 = 23$, so the {\kl} implies $42 > 65-23=42$, which is a contradiction.
Thus, it must be that $x_{i-2} = x_{j+2}$ and $x_{j-2}=x_{i+2}$, so $k = 8$.

Without loss of generality, let $i = 0$ and $j = 4$.
In this case, $w_G(N[x_0,x_3]) = 6 \cdot 4 + 5 \cdot 3 = 39$ and $c_G(N[x_0,x_3]) \le 3 \cdot 3 + 1 \cdot 2 = 11$, so the {\kl} implies $28>39-11$, which is a contradiction.
\end{proof}

Let $U=N[X\cup Y]$, so  $|U| = 3|X| + 3|Y_1| + 2|Y_2|$ and $w_G(U) = 17|X| + 16|Y_1| + 11|Y_2|$.
Note that $G[U]$ contains $C$ as a subgraph and every edge in $\partial(U)$ is incident with a vertex in $X^*\cup Y^*$. 
By the {\kl}, $14|X| + 14|Y_1| + 14|Y_2| = 14|X \cup Y| > (17|X| + 16|Y_1| + 11|Y_2|) - c_G(U)$,  so  \begin{eqnarray}\label{eq:C}
&&c_G(U)>3|X|+2|Y_1|-3|Y_2|.
\end{eqnarray}

Define the initial charge $\mu(u)$ for every vertex $u\in X^*\cup Y^*\cup (V(G)\setminus U)$ as below:
\[\mu(u) = \begin{cases}
 -3 &\text{ if } u \in X^*,\\
 -2 &\text{ if } u \in Y^*,\\
 w_{G-U}(u)-w_{G}(u) &\text{ if } u \in V(G)\setminus U. 
\end{cases}\]
Note that the sum of the initial charge is $c_G(U)-3|X^*|-2|Y^*|=c_G(U)-3(|X|-|Y_2|)-2|Y_1|$, which is positive by \eqref{eq:C}. 
We distribute the charge by the following rule, so the sum of the total charge is preserved. 
Recall that every edge in $\partial(U)$ is incident with a vertex in $X^*\cup Y^*$.

\smallskip

\noindent
\textbf{[Rule]} Let $uv\in \partial(U)$ where $u\in U$.
If $u\in X^*$, then $u$ sends $-3$ to $v$. If $u\in Y^*$, then $u$ sends $-1$ to $v$.

\smallskip

Let $\mu^*(u)$ be the final charge of each vertex $u$ after applying the above rule. 
We will obtain a contradiction by proving that the final charge of each vertex in $X^*\cup Y^*\cup (V(G)\setminus U)$ is non-positive. 

\begin{claim}\label{clm:final_ch}
$\mu^*(v)\le 0$ for a vertex $v\in X^*\cup Y^*\cup (V(G)\setminus U)$.
\end{claim}
\begin{proof}
Since an $X^*$-vertex is incident with at most one edge of $\partial(U)$ and a $Y^*$-vertex  is incident with at most two edges of $\partial(U)$,  it is clear from the definitions that
$\mu^*(v)\le 0$ for a vertex $v\in X^*\cup Y^*$.
Suppose to the contrary that there is a vertex $v\in V(G)\setminus U$ such that $\mu^*(v)>0$.
Clearly, $v$ is incident with a vertex in $X^*\cup Y^*$.  
By [{\bf{Rule}}], we have
\begin{equation}\label{eq:charge}
w_{G-U}(v)-w_{G}(v) > 3|X^*\cap N(v)|+|Y^*\cap N(v)|. 
\end{equation}
We divide the proof into two cases.

\smallskip
\noindent\textbf{(Case 1)} 
Suppose that $v$ has no $Y^*$-neighbor, so $v$ must have an $X^*$-neighbor.
If $|X^*\cap N(v)|=1$, then  $w_{G-U}(v)-w_G(v) \le 3 = 3|X^*\cap N(v)|$, which is a contradiction to \eqref{eq:charge}.
If $|X^*\cap N(v)|=3$, then $v$ is a $3$-vertex in $G$, and $w_{G-U}(v)-w_{G}(v) = 9= 3|X^*\cap N(v)|$, which is a contradiction to \eqref{eq:charge}.
Thus $|X^*\cap N(v)|=2$.
If $v$ is a $3$-vertex, then $w_{G-U}(v)-w_G(v) = 4 <6=3|X^* \cap N(v)|$, which is a contradiction to \eqref{eq:charge}.
Thus, $v$ is a $2$-vertex with only $X^*$-neighbors, which is a contradiction to Claim~\ref{clm:two:C}.

\smallskip
\noindent\textbf{(Case 2)} 
Suppose that $v$ has a $Y^*$-neighbor.
By the definition of $Y^*$, we see that $v$ is a $3$-vertex.
By Claim~\ref{clm:z2}, $v$ has exactly one $Y^*$-neighbor. 
If $|X^* \cap N(v)| =0$, 
then $w_{G-U}(v)-w_G(v) = 1 =3|X^*\cap N(v)|+|Y^*\cap N(v)|$, which is a contradiction to \eqref{eq:charge}.
If $|X^* \cap N(v)| =1$, 
then $w_{G-U}(v)-w_G(v) = 4 = 3|X^*\cap N(v)|+|Y^*\cap N(v)|$, which is a contradiction to \eqref{eq:charge}.
Thus $|X^* \cap N(v)| = 2$, which implies that $v$ is an $A_0$-vertex, since $X^* \subseteq A_2$ and $Y^* \subseteq A_1$.
This is a contradiction since an $A_0$-vertex is not adjacent to an $A_2$-vertex.
\end{proof}

By Claim~\ref{clm:final_ch}, the final charge of each vertex is non-positive, and we reach a contradiction to \eqref{eq:C}.
Hence, the counterexample $G$ cannot exist. 

\section{Ratio of independent domination and  domination}\label{sec:ratio}

In this section, we prove  Theorem~\ref{thm:ratio}.
Given a graph $G$ and $v\in X \subseteq V(G)$, the {\it $X$-external private neighborhood} of a vertex $v$ is defined by $\epn(v,X) = \{u \in V(G)\setminus X \mid N_G[u] \cap X = \{v\}\}$.
A dominating set $X$ of $G$ inducing a graph with maximum degree at most $1$ is  a {\it near independent dominating set.}

\begin{lemma}\label{clm:igd}
For a near independent dominating set $X$ of a cubic graph $G$, $i(G) \le |X|+ \frac{1}{2}n_1(G[X])$.
\end{lemma}
\begin{proof}
We use induction on $n_1(G[X])$.
 If $n_1(G[X])=0$, then $X$ is an independent dominating set of $G$, so the statement holds since $i(G) \le |X|$.

Now, assume  $n_1(G[X])>0$.
Take a vertex $v \in X$ with a neighbor in $X$. 
Since $G$ is cubic, $|\epn(v,X)| \le 2$.
Take a maximal independent set $I$ of $G$ such that $I\subseteq \epn(v,X)$, and let $X' = (X \cup I) \setminus \{v\}$. (Note that $I=\emptyset$ when $\epn(v,X)=\emptyset$.)
Then  $X'$ is a near independent dominating set of $G$ such that $|X'|\le|X|+1$ and $n_1(G[X']) \le n_1(G[X])-2$, so by the inductive hypothesis, we have $i(G) \le |X'|+\frac{1}{2}n_1(G[X']) \le (|X| +1) + \frac{1}{2}(n_1(G[X])-2) = |X|+ \frac{1}{2}n_1(G[X])$.
\end{proof}

\begin{proof}[Proof of Theorem~\ref{thm:ratio}]
Let $G$ be a cubic graph without 4-cycles, and let
 $D$ be a minimum dominating set of $G$ with the minimum number of edges in $G[D]$. 
Suppose that a vertex $v\in D$ has two neighbors in $D$. 
By the minimality of $|D|$, we have $|\epn(v,D)|=1$. 
Then $D'=(D \cup \epn(v,D)) \setminus \{v\}$ is a dominating set of $G$.
Note that $|D'|=|D|$ and $|E(G[D'])|<|E(G[D])|$, which is a contradiction to the choice of $D$. 
Thus, every vertex in $D$ has at most one neighbor in $D$, so $D$ is a near independent dominating set of $G$. 

Let $m_i=n_i(G[D])$. Then
\begin{eqnarray}\label{eq:D}
&&\gamma(G)=|D|=m_0+m_1.
\end{eqnarray}
If $m_0 \ge m_1$, 
then \eqref{eq:D} implies $\gamma(G) \ge 2 m_1$, so   
together with Lemma~\ref{clm:igd}, we have the desired inequality since 
\[ i(G) \le |D|+\frac{1}{2}  m_1\le \gamma(G)+ \frac{1}{4}\gamma(G)=\frac{5}{4}\gamma(G).\]
If $m_0 \le m_1$, then \eqref{eq:D} implies $\gamma(G)\ge 2m_0$,  so by Theorem~\ref{mainthm:cubic}, it holds that
\[i(G) \le \frac{5}{14}|V(G)| \le \frac{5}{14}\left(4 m_0 + 3m_1\right) = \frac{5}{14}\left(3\gamma(G)+m_0  \right) 
\le \frac{5}{14}\left(3\gamma(G)+\frac{1}{2}\gamma(G)\right)=\frac{5}{4}\gamma(G),\] where the second inequality is from the fact that $D$ is a dominating set of a cubic graph, and the third equality is from \eqref{eq:D}.
\end{proof}

\section*{Acknowledgements}
Eun-Kyung Cho was supported by Basic Science Research Program through the National Research Foundation of Korea (NRF) funded by the Ministry of Education (NRF-2020R1I1A1A0105858711).
Ilkyoo Choi was supported by the Basic Science Research Program through the National Research Foundation of Korea (NRF) funded by the Ministry of Education (NRF-2018R1D1A1B07043049), and also by the Hankuk University of Foreign Studies Research Fund.
Boram Park and Hyemin Kwon were supported by Basic Science Research Program through the National Research Foundation of Korea (NRF) funded by the Ministry of Science, ICT and Future Planning (NRF-2018R1C1B6003577).

\bibliographystyle{abbrv}
\bibliography{ref}

\end{document}